\documentclass[11pt]{article}

\RequirePackage[amsthm,amsmath]{imsart}

\usepackage[mathscr]{eucal}
\usepackage{graphicx}
\usepackage{latexsym}
\usepackage{amssymb}
\usepackage{psfrag}
\usepackage{epsfig}
\usepackage{enumerate}
\DeclareMathAlphabet{\mathpzc}{OT1}{pzc}{m}{it}

\include{psbox}

\usepackage{amsfonts}
\usepackage{graphicx}
\usepackage{amsfonts}
\usepackage{color}
\usepackage[usenames,dvipsnames]{xcolor}
\usepackage[textsize=small]{todonotes}

\usepackage[numbers]{natbib}

\begin{document}

\newtheorem{theorem}{\bf Theorem}[section]

\newtheorem{definition}{\bf Definition}[section]
\newtheorem{corollary}{\bf Corollary}[section]

\newtheorem{lemma}{\bf Lemma}[section]
\newtheorem{assumption}{Assumption}
\newtheorem{condition}{\bf Condition}[section]
\newtheorem{proposition}{\bf Proposition}[section]
\newtheorem{definitions}{\bf Definition}[section]
\newtheorem{problem}{\bf Problem}
\numberwithin{equation}{section}
\newcommand{\skp}{\vspace{\baselineskip}}
\newcommand{\noi}{\noindent}
\newcommand{\osc}{\mbox{osc}}
\newcommand{\lfl}{\lfloor}
\newcommand{\rfl}{\rfloor}

\theoremstyle{remark}
\newtheorem{example}{\bf Example}[section]
\newtheorem{remark}{\bf Remark}[section]

\newcommand{\img}{\imath}
\newcommand{\iy}{\infty}
\newcommand{\eps}{\varepsilon}
\newcommand{\del}{\delta}
\newcommand{\Rk}{\mathbb{R}^k}
\newcommand{\RR}{\mathbb{R}}
\newcommand{\spa}{\vspace{.2in}}
\newcommand{\V}{\mathcal{V}}
\newcommand{\E}{\mathbb{E}}
\newcommand{\I}{\mathbb{I}}
\newcommand{\PP}{\mathbb{P}}
\newcommand{\sgn}{\mbox{sgn}}
\newcommand{\ti}{\tilde}
\newcommand{\vs}{\varsigma}
\newcommand{\vr}{\varrho}
\newcommand{\ups}{\upsilon}

\newcommand{\QQ}{\mathbb{Q}}

\newcommand{\XX}{\mathbb{X}}
\newcommand{\XXz}{\mathbb{X}^0}
\newcommand{\MM}{\mathbb{M}}
\newcommand{\lan}{\langle}
\newcommand{\ran}{\rangle}
\newcommand{\lf}{\lfloor}
\newcommand{\rf}{\rfloor}
\def\wh{\widehat}
\newcommand{\defn}{\stackrel{def}{=}}
\newcommand{\txb}{\tau^{\epsilon,x}_{B^c}}
\newcommand{\tyb}{\tau^{\epsilon,y}_{B^c}}
\newcommand{\tilxb}{\tilde{\tau}^\eps_1}
\newcommand{\pxeps}{\mathbb{P}_x^{\eps}}
\newcommand{\non}{\nonumber}
\newcommand{\dist}{\mbox{dist}}

\newcommand{\Om}{\mathnormal{\Omega}}
\newcommand{\om}{\omega}
\newcommand{\vph}{\varphi}
\newcommand{\Del}{\mathnormal{\Delta}}
\newcommand{\Gam}{\mathnormal{\Gamma}}
\newcommand{\Sig}{\mathnormal{\Sigma}}

\newcommand{\tilyb}{\tilde{\tau}^\eps_2}
\newcommand{\beq}{\begin{eqnarray*}}
\newcommand{\eeq}{\end{eqnarray*}}
\newcommand{\beqn}{\begin{eqnarray}}
\newcommand{\eeqn}{\end{eqnarray}}
\newcommand{\ink}{\rule{.5\baselineskip}{.55\baselineskip}}

\newcommand{\bt}{\begin{theorem}}
\newcommand{\et}{\end{theorem}}
\newcommand{\deps}{\Del_{\eps}}
\newcommand{\dbl}{\mathbf{d}_{\tiny{\mbox{BL}}}}

\newcommand{\be}{\begin{equation}}
\newcommand{\ee}{\end{equation}}
\newcommand{\ac}{\mbox{AC}}
\newcommand{\BB}{\mathbb{B}}
\newcommand{\VV}{\mathbb{V}}
\newcommand{\DD}{\mathbb{D}}
\newcommand{\KK}{\mathbb{K}}
\newcommand{\HH}{\mathbb{H}}
\newcommand{\TT}{\mathbb{T}}
\newcommand{\CC}{\mathbb{C}}
\newcommand{\Z}{\mathbb{Z}}
\newcommand{\SSS}{\mathbb{S}}
\newcommand{\EE}{\mathbb{E}}
\newcommand{\NN}{\mathbb{N}}

\newcommand{\clg}{\mathcal{G}}
\newcommand{\clb}{\mathcal{B}}
\newcommand{\cls}{\mathcal{S}}
\newcommand{\clc}{\mathcal{C}}
\newcommand{\clj}{\mathcal{J}}
\newcommand{\clm}{\mathcal{M}}
\newcommand{\clx}{\X}
\newcommand{\cld}{\mathcal{D}}
\newcommand{\cle}{\mathcal{E}}
\newcommand{\clv}{\mathcal{V}}
\newcommand{\clu}{\mathcal{U}}
\newcommand{\clr}{\mathcal{R}}
\newcommand{\clt}{\mathcal{T}}
\newcommand{\cll}{\mathcal{L}}
\newcommand{\clz}{\mathcal{Z}}

\newcommand{\cli}{\mathcal{I}}
\newcommand{\clp}{\mathcal{P}}
\newcommand{\cla}{\mathcal{A}}
\newcommand{\clf}{\mathcal{F}}
\newcommand{\clh}{\mathcal{H}}
\newcommand{\N}{\mathbb{N}}
\newcommand{\Q}{\mathbb{Q}}
\newcommand{\bfx}{{\boldsymbol{x}}}
\newcommand{\bfh}{{\boldsymbol{h}}}
\newcommand{\bfs}{{\boldsymbol{s}}}
\newcommand{\bfm}{{\boldsymbol{m}}}
\newcommand{\bff}{{\boldsymbol{f}}}
\newcommand{\bfb}{{\boldsymbol{b}}}
\newcommand{\bfw}{{\boldsymbol{w}}}
\newcommand{\bfz}{{\boldsymbol{z}}}
\newcommand{\bfu}{{\boldsymbol{u}}}
\newcommand{\bfell}{{\boldsymbol{\ell}}}
\newcommand{\bfn}{{\boldsymbol{n}}}
\newcommand{\bfd}{{\boldsymbol{d}}}
\newcommand{\bfbeta}{{\boldsymbol{\beta}}}
\newcommand{\bfzeta}{{\boldsymbol{\zeta}}}
\newcommand{\bfnu}{{\boldsymbol{\nu}}}

\newcommand{\curvz}{{\bf \mathpzc{z}}}
\newcommand{\curvx}{{\bf \mathpzc{x}}}
\newcommand{\curvi}{{\bf \mathpzc{i}}}
\newcommand{\curvs}{{\bf \mathpzc{s}}}
\newcommand{\blip}{\mathbb{B}_1}

\newcommand{\BM}{\mbox{BM}}

\newcommand{\tac}{\mbox{\scriptsize{AC}}}
\newcommand{\beginsec}{
\setcounter{lemma}{0} \setcounter{theorem}{0}
\setcounter{corollary}{0} \setcounter{definition}{0}
\setcounter{example}{0} \setcounter{proposition}{0}
\setcounter{condition}{0} \setcounter{assumption}{0}
\setcounter{remark}{0} }

\numberwithin{equation}{section} \numberwithin{lemma}{section}

\begin{frontmatter}
\title{Large deviations for Multidimensional State-Dependent  Shot Noise Processes}

 \runtitle{Large deviations for Poisson shot noise process}

\begin{aug}
\author{ Amarjit Budhiraja and Pierre Nyquist
\\ \ \\
}
\end{aug}


\skp

\begin{abstract}
Shot noise processes are used in applied probability to model a variety of physical systems in, for example, teletraffic theory, insurance and risk theory and in the engineering sciences. In this work we prove a large deviation principle for the sample-paths of a general class of multidimensional state-dependent Poisson shot noise processes. The result covers previously known large deviation results for one dimensional state-independent shot noise processes with light tails. We use the weak convergence approach to large deviations, which reduces the proof to establishing the appropriate convergence of certain controlled versions of the original processes together with relevant results on existence and uniqueness. 

\noi {\bf AMS 2010 subject classifications:} 60F10, 60G55, 60K30

\noi {\bf Keywords:} Large deviations; Poisson shot noise; Poisson random measure; variational representations
\end{abstract}

\end{frontmatter}


\def\N{\mathbb{N}}
\def\P{\mathbb{P}}
\def\F{\mathcal{F}}
\def\R{\mathbb{R}}
 \def\X{\mathbb{X}}
 \def\Y{\mathbb{Y}}
\def\S{\mathbb{S}}
\def\W{\mathbb{W}}
\def\V{\mathbb{V}}
\def\U{\mathbb{U}}
\def\C{\mathbf{C}}
\def\B{\mathcal{B}}
\def\Lexp{\mathcal{L}_{\mbox{\tiny{exp}}}}
\def\strategy{h}
\def\Prob{\mathbb{P}}	
\def\Q{\mathbb{Q}}		
\def\Eproc{\mathbb{G}}
\def\E{\mathbb{E}}			
\def\Eb{\overline{\mathbb{E}}}
\def\calR{\mathcal{R}}
\def\frakF{\mathfrak{F}}
\def\linf{\ell _{\infty}(\mathfrak{F})}
\def\lin{\ell _{\infty}}
\newcommand{\edfis}{\tilde{\mathbf{F}}}
\newcommand{\empd}{\mathbf{F}}
\newcommand{\Rinf}{\mathbb{R}^{\infty}}	
\newcommand{\M}{\mathcal{M}}	
\newcommand{\bM}{\mathbb{M}}	
\newcommand{\re}{\mathcal{H}}	
\newcommand{\G}{\mathcal{G}}	
\newcommand{\A}{\mathcal{A}}	
\newcommand{\Var}{\textrm{Var}}
\newcommand{\Exp}{\mathbb{E}}
\newcommand{\calF}{\mathcal{F}}	
\newcommand{\calE}{\mathcal{E}}	
\newcommand{\lebT}{\lambda _{T}}	
\newcommand{\lebinf}{\lambda _{\infty}}	
\newcommand{\cadlag}{c\`{a}dl\`{a}g }

\maketitle  
\section{Introduction}
\label{sec:intro}
The goal of the current work is to study large deviation results for a general family of multidimensional state-dependent shot noise processes. Shot noise processes provide a natural class of models for systems in which (some aspect of) the state of the system is determined by the arrival of shocks. A typical application is in the context of queueing systems, in which the arrival of customers can be interpreted as shocks and one is interested in, say, the current workload or cost incurred - due to performed work - by current and former customers. Another common area of application is insurance, the shocks being claims that arrive according to an underlying point process.

Due to their usefulness in describing various physical systems shot noise processes have been studied extensively, both theoretically as well as from the perspective of applications. For some general treatments of this class of processes and their properties see \cite{Rice1945, Rice1977, Daley1971, DoneyObrien1991}; various asymptotic properties of shot noise are found in \cite{Lane1984, LundMcCXiao, McCormick1997, Samorodnitsky98, KMS2003}, the last three of which deal with heavy-tailed phenomena. An example of shot noise processes in the queueing context is provided in \cite{Huffer1987}, whereas \cite{KM1995, KMS2003} considers applications to insurance and risk theory. Another type of application is to storage processes, see, e.g., \cite{Lund1996, BrockResTw}.

In this paper we are concerned with Poisson shot noise processes, that is the underlying point process governing the arrivals of the shocks is a Poisson process. Large deviations for a family of Poisson shot noise processes have been studied in \cite{GaneshMacciTorrisi2005} and the precise result therein is as follows.
Let $N$ be a homogeneous Poisson process with unit intensity  and let $Z_1, Z_2,\dots$, be independent and identically distributed $\X$-valued random variables, each with distribution $\nu$ and independent of $N$. Here $\X$ is some locally compact Polish space.  For a function $H : \R_+ \times  \X \to \R_+$, referred to as the shot shape, consider the Poisson shot noise process $\{ X(t); t\in [0,T]\}$, defined as
\begin{align}
\label{eq:procX}
	X(t) = \sum_{n = 1}^{N(t)} H(t-T_n, Z_n),
\end{align}
where $T_1, T_2, \ldots $ are jump instants of $N$. Suppose that 
\begin{equation}
	\mbox{ for each } z \in \X, \  t\mapsto H(t,z) \mbox{ is nondecreasing and c\`{a}dl\`{a}g and } H(0,z) = 0.
	\label{eq:eq619}
\end{equation}
	 Let $h(z) = \lim_{t\to \infty}H(t,z)$.  The function $h$ is referred to as the shot value for the shot noise process $X$.
Suppose that $h$ satisfies the following condition:
\begin{equation}
	\label{eq:eq601/26}
	\mbox{ for every } \vartheta \in \R, \ \int_{\X} e^{\vartheta h(z)} \nu(dz) < \infty .
\end{equation}
For $\eps > 0$, let $X^{\eps}(t) = \eps X(\eps^{-1}t)$, $t \in [0,T]$.  
The paper
 \cite{GaneshMacciTorrisi2005} shows that $\{X^{\eps}\}_{\eps > 0}$ satisfies  a large deviation principle (LDP) in $D ([0,T] : \RR_+)$ as $\eps \to 0$, where $D ([0,T] : \RR_+)$ denotes the space of c\`{a}dl\`{a}g functions from $[0,T]$ to $\RR_+$ which is equipped with the usual Skorohod topology.
 
The goal of this work is to study large deviation properties of general state-dependent multidimensional shot noise processes.  Such processes are natural models for systems where
the impact of a shot depends on the current state value of the system.  
%
%
%
%
In order to prove large deviation results we use the fact that a Poisson shot noise process can be represented as an integral with respect to a Poisson random measure. Using such representations, our work builds on certain variational formulas for functionals of a Poisson random measure \cite{BudDupMar2011} and their application to large deviations \cite{BudChenDup2013}. Rather than traditional large deviation techniques, such as those used in \cite{GaneshMacciTorrisi2005}, we  use the weak convergence approach to large deviations. With the results of \cite{BudDupMar2011, BudChenDup2013} this amounts to proving the appropriate convergence of certain controlled versions of the original process, together with the necessary existence and uniqueness results. The main advantage of the weak convergence approach is that it avoids the discretization/approximation arguments and exponential estimates typically encountered in a large deviation analysis (see e.g. \cite{Dembo98}). Such approximation methods are used extensively in \cite{GaneshMacciTorrisi2005} and in general are  difficult to implement for complex settings such as the state-dependent shot noise processes considered here. 

Large deviation results such as those considered in this work can be used to determine the most likely path to rare events. In applications of shot noise processes these results can thus be used to try and prevent unwanted behaviour of the system in question. Moreover, large deviation results can be used to design efficient Monte Carlo algorithms. For an example of such simulations in the context of shot noise processes see \cite{Torrisi2004}.
Applications of our results to rare-event simulation problems will be studied in our future work.


We now introduce the   multidimensional state-dependent shot noise processes that will be studied in this work.  Let
for each $\eps > 0$, $H_{\eps}: \R_+ \times \X \times \R_+^d \to \R_+^d$ be a measurable function and  consider the stochastic
process $\tilde X^{\eps}$ given as the solution of the following equation,
\begin{equation}
	\label{eq:eq609/26}
		\ti X^{\eps}(t) = \sum_{n = 1}^{N(t)} H_{\eps}(t-T_n, Z_n, \tilde X^{\eps}(T_n-)), \; t \ge 0.
\end{equation}
Let $X^{\eps}(t) = \eps \ti X^{\eps}(\eps^{-1}t)$.  We will give a sufficient condition on the collection of maps
$\{H_{\eps}\}_{\eps > 0}$ under which $\{X^{\eps}\}_{\eps > 0}$ satisfies a LDP in $D([0,T]: \RR _+ ^d)$.  It will be convenient to work
with the following, equivalent in law, representation for $X^{\eps}$.
Let $\bfn_{\eps}$ be a Poisson random measure (PRM) on $\X_T = [0,T]\times \X$ with intensity $\eps^{-1}\nu_T = \eps^{-1}\lambda \otimes \nu$
where $\lambda$ is the Lebesgue measure on $[0,T]$.  Let $X^{\eps}$ solve the equation
\begin{equation}
	\label{eq:eq627/26}
	X^{\eps}(t) = \eps\int_{\X_t} H_{\eps}(\eps^{-1}(t-s), z, \eps^{-1}X^{\eps}(s-)) \bfn_{\eps}(ds\, dz),\; t \in [0, T].
\end{equation}
where for $t \in [0,T]$, $\X_t = [0,t]\times \X$.
It is easy to check that $X^{\eps}$ defined in \eqref{eq:eq627/26} and $\eps \ti X^{\eps}(\eps^{-1}\cdot)$, where $\ti X^{\eps}$ is
as in \eqref{eq:eq609/26}, have the same distribution. 
We will in fact consider a more general setting in that we will study the collection $\{X^{\eps}\}_{\eps > 0}$ given as the solution of the equation \eqref{eq:eq627/26} where the measure $\nu$ describing the intensity of $\bfn_{\eps}$
is a general $\sigma$-finite measure on $(\X, \B(\X))$. This allows for a non-integrable number of shocks on a bounded time interval.
One can formulate general sufficient conditions under which \eqref{eq:eq627/26} has a unique pathwise solution.  We will instead take unique solvability of the
equation as one of our  basic assumptions (see Condition \ref{cond:uniqsolv}). In Section \ref{sec:two-1} we introduce our assumptions (Conditions \ref{cond:uniqsolv} and
\ref{cond:mainonh}) on $H_{\eps}$. Our main result (Theorem \ref{thm:main}) shows that under these conditions, $X^{\eps}$ given as the solution
of the stochastic integral equation \eqref{eq:eq627/26b} satisfies a LDP in $D([0,T]: \R_+^d)$ as $\eps \to 0$.  
The LDP established in \cite{GaneshMacciTorrisi2005} is an immediate consequence of Theorem \ref{thm:main}.  

The rest of the paper is organized as follows.  
Section \ref{sec:three} contains the proof of wellposedness of an ODE associated
with the asymptotics of controlled analogues of \eqref{eq:eq627/26} (Theorem \ref{thm:odewell}).
In Section \ref{sec:four} we recall a result from \cite{BudDupMar2011} that gives a general sufficient condition (Condition \ref{cond:gensuff})
for a LDP to hold for measurable functionals of PRM.  Theorem \ref{thm:main} is proved by verifying this sufficient condition.  Part (a) of the condition
is verified in Section \ref{sec:five-1} while part (b)  is considered in Section \ref{sec:five-2}.

The following notation will be used.
The space of probability measures on a Polish space $\SSS$, equipped with the topology of weak convergence, will be denoted by $\clp(\SSS)$.  
For a function $f: [0,T] \to \RR^k$, set $\|f\|_{*,t} = \sup_{0\le s \le t}\|f(s)\|$, $t \in [0,T]$.  
For $\vartheta \in \clp(\SSS)$ and a $\vartheta$ -integrable $f$ on $\SSS$, we denote $\int_{\SSS} f(x) \vartheta(dx)$ as $\langle f, \vartheta\rangle$.
The Borel $\sigma$-field on a Polish space $\SSS$ will be denoted as $\clb(\SSS)$. The space of functions that are right-continuous with left limits (RCLL) from $[0, \infty)$ [resp. $[0,T]$] to $\SSS$
will be denoted as $D([0,\infty) : \SSS)$ [resp. $D([0,T]: \SSS)$] and are equipped with the usual Skorohod topology.
For a bounded function $f$ from $\SSS$ to $\RR$, we denote $\|f\|_{\infty} = \sup_{x \in \SSS}|f(x)|$.
Convergence of a sequence $\{X_n\}$ of $\SSS$-valued random variables in distribution to $X$ will be written as $X_n \Rightarrow X$.

For a $\sigma$-finite measure $\nu$ on a Polish space $\SSS$, $\cll^p_{\RR^k}(\SSS, \nu)$ will denote the  space of $p$-integrable functions, with respect to $\nu$, from $\SSS$ to $\RR^k$.  When $k=1$, we will merely write $\cll^p(\SSS, \nu)$
or $\cll^p(\nu)$.
We will usually denote by $\kappa, \kappa_1, \kappa_2, \cdots$, the constants that appear in various estimates within a proof.  The values of these constants may change from one proof to another.

\section{Main Result.}
\label{sec:two}
Our basic collection of stochastic processes $\{X^{\eps}\}_{\eps > 0}$ is given in terms of Poisson random measures $\{\bfn_{\eps}\}_{\eps > 0}$.  We will like all these Poisson random measures 
to be defined on a common probability space.  For this, the following construction will be useful.
Let $\mathbb{Y}=\mathbb{X}\times \lbrack 0,\infty )$ and $\mathbb{Y}%
_{T}=[0,T]\times \mathbb{Y}$. 
For a locally compact Polish space $\Z$, let $\mathcal{M}_{FC}(\mathbb{Z})$ be the space of all measures $\nu$ on $(\mathbb{Z}, \mathcal{B}(%
\mathbb{Z}))$ such that $\nu(K) < \infty$ for every compact $K$ in $\mathbb{Z}$.  This space is equipped with the usual topology of vague convergence.
Let $\bar{\mathbb{M}}=\mathcal{M}_{FC}(\mathbb{Y}_{T})$ and let $\bar{\mathbb{P}}$ be the unique probability measure on $(%
\bar{\mathbb{M}},\mathcal{B}(\bar{\mathbb{M}}))$ under which the canonical
map, $\bar{N}:\bar{\mathbb{M}}\rightarrow \bar{\mathbb{M}},\ \bar{N}(m) =
m $, is a Poisson random measure with intensity measure $\bar{\nu}%
_{T} = \lambda\otimes \nu \otimes \lambda _{\infty }$, where $\lambda
_{\infty }$ is the Lebesgue measure on $[0,\infty )$. The corresponding expectation
operator will be denoted by $\bar{\mathbb{E}}$. Let $\mathcal{F}_{t} =
\sigma \{\bar{N}((0,s]\times A):0\leq s\leq t,A\in \mathcal{B}(\mathbb{Y})\}$%
, and let $\bar{\mathcal{F}}_{t}$ denote the completion under $\bar{\mathbb{P%
}}$. We denote by $\bar{\mathcal{P}}$ the predictable $\sigma $-field on $%
[0,T]\times \bar{\mathbb{M}}$ with the filtration $\{\bar{\mathcal{F}}%
_{t}:0\leq t\leq T\}$ on $(\bar{\mathbb{M}},\mathcal{B}(\bar{\mathbb{M}}))$.

For $\eps > 0$, let $N^{\eps^{-1}}$ be a counting process on $\X_T$ defined as
\begin{equation}
N^{\eps^{-1} }((0,t]\times U)=\int_{(0,t]\times U \times \R_+}1_{[0,\eps^{-1}]}(r)\bar{N}(ds\, dx\, dr),\quad t\in \lbrack 0,T], \ U\in
\mathcal{B}(\mathbb{X}).  \label{Eqn:nocontrol}
\end{equation}
Clearly $N^{\eps^{-1}}$ has the same distribution as $\bfn_{\eps}$, i.e. it is a PRM on $\X_T = [0,T]\times \X$ with intensity $\eps^{-1}\lambda \otimes \nu$.

Next we introduce our assumptions and present the main result.
\subsection{Assumptions.}
\label{sec:two-1}
Our first assumption is on the unique solvability of \eqref{eq:eq627/26}.  
Note that $N^{\eps^{-1}}$ is a $\mathbb{M}$-valued random variable, where $\mathbb{M} = \mathcal{M}_{FC}(\mathbb{X}_T)$. 
\begin{condition}
	\label{cond:uniqsolv}
	For each $\varepsilon>0$ there is a measurable map ${\mathcal{G}}^{\varepsilon}:\MM \to D([0,T]: \R_+^d)$
	such that for any probability space $(\tilde{\Omega},\tilde{\mathcal{F}%
	},\tilde{P})$ on which is given a 
	Poisson random measure ${\tilde{\mathbf{n}}}_{\varepsilon}$ on $\X_{T}$ with intensity measure $\varepsilon
	^{-1}\nu_{T}$, ${X}^{\varepsilon}= {\mathcal{G}}^{\varepsilon}
	(\varepsilon\tilde{\mathbf{n}}_{\varepsilon})$ is a $\tilde{\mathcal{F}}%
	_{t}=\sigma\{\tilde{\mathbf{n}}_{\varepsilon}(B\times\lbrack0,s]),s\leq
	t,B\in\mathcal{B}(\X),\nu(B)<\infty\}$-adapted c\`{a}dl\`{a}g process that is the
	unique solution of the stochastic integral equation
	\eqref{eq:eq627/26}. 
\end{condition}
Condition \ref{cond:uniqsolv} is satisfied quite generally.  For example, if $\nu$ is a finite measure, the unique solvability
is immediate from a recursive construction of a solution of \eqref{eq:eq627/26} from one jump to the next.  For more general $\nu$, Condition \ref{cond:uniqsolv} will hold under suitable Lipschitz and growth assumptions on $H_{\eps}$ (cf. Theorem III.2.3.2 of
\cite{JacodShiryaev}).
The condition in particular says that  $X^{\varepsilon}={\mathcal{G}}^{\varepsilon}(\varepsilon
N^{\varepsilon^{-1}})$ is the unique solution of
\begin{equation}
	\label{eq:eq627/26b}
	X^{\eps}(t) = \eps\int_{\X_t} H_{\eps}(\eps^{-1}(t-s), z, \eps^{-1}X^{\eps}(s-)) N^{\eps^{-1}}(ds\, dz),\; t \in [0, T].
\end{equation}
on $(\bar{\mathbb{M}},\mathcal{B}(\bar{\mathbb{M}}), \bar{\P})$.  For the rest of this work $X^{\eps}$ will denote
the solution of \eqref{eq:eq627/26b}.

Next, we introduce our second main assumption on the family $\{H_{\eps}\}_{\eps > 0}$.
We denote by $\Lexp$ the family of all measurable functions $r: \X \to \R_+$ such that whenever $A \in \B(\X)$ is such that $\nu(A)< \infty$,
$$\int_{A} e^{\vartheta r(z)} \nu(dz) < \infty, \; \mbox{ for all } \vartheta \in \R.$$
Note that if $\nu$ is a probability measure this condition merely says that $r(Z)$ has an everywhere finite moment generating
function, where $Z$ is a random variable with probability distribution $\nu$.
\begin{condition}
	\label{cond:mainonh}
	There are measurable functions $\bar H$ and $R_{\eps}$ from $\R_+\times \X \times \R_+^d$ to $\R_+^d$
	and $\R^d$ respectively;  $\vs_{\eps}, \vs$ from $\X$ to $\R_+$; and  $h$ from $\X \times \R_+^d$ to $\R_+^d$ such that the following hold.
	\begin{enumerate}[(a)]
		\item For $\eps > 0$ and $(t,z,x) \in \R_+\times \X \times \R_+^d$
		$$H_{\eps}(t,z,x) = \bar H(t,z, \eps x) + R_{\eps}(t,z, \eps x).$$
		\item For  $\eps > 0$ and $(z,x) \in  \X \times \R_+^d$
		$$\sup_{t\ge 0} \|R_{\eps}(t,z,x)\| \le \vs_{\eps}(z)(\|x\| + 1).$$
		\item $\vs_{\eps} \le \vs$, $\vs \in \Lexp \cap \cll^1(\nu)$ and for $\nu$ a.e. $z \in \XX$, $\vs_{\eps}(z)\to 0$ as $\eps \to 0$.
		\item For $(z,x) \in  \X \times \R_+^d$, $t \mapsto \bar H(t,z,x)$ is \cadlag and non-decreasing (coordinate-wise) and $\bar H(0, z,x) = 0$.
		\item For every $z \in \X$ and $m > 0$, $\sup_{\|x\| \le m} |\bar H(t,z,x) - h(z,x)\| \to 0$ as $t \to \infty$.
		\item For some $L_h \in \Lexp \cap \mathcal{L}^1(\nu)$, 
		$$\|h(z,x) - h(z,x')\| \le L_h(z) \|x-x'\|, \mbox{ for all } x,x' \in \R_+^d, \ z \in \X.$$
		\item For some $M_h \in \Lexp \cap \mathcal{L}^1(\nu)$
		$$\|h(z,x)\| \le M_h(z) (1+ \|x\|), \mbox{ for all } x \in \R_+^d, \ z \in \X.$$
	\end{enumerate}
\end{condition}


The setting considered in \cite{GaneshMacciTorrisi2005} corresponds to the case where $H_{\eps}(t,z,x)$ is independent of
$\eps$ and $x$; in particular $R_{\eps}\equiv 0$.  Conditions \ref{cond:uniqsolv} and \ref{cond:mainonh} will be standing assumptions for this work and will not be explicitly mentioned in the statements of results.

\subsection{Controlled ODE.}
\label{sec:contode}
In this section we will consider an ODE that arises in the asymptotic analysis of the controlled analogues of  \eqref{eq:eq627/26b}.
Define $l:[0,\infty
)\rightarrow \lbrack 0,\infty )$ by
\begin{equation*}
l(r)=r\log r-r+1,\quad r\in \lbrack 0,\infty ).
\end{equation*}%
For $g:\X_T \to [0, \infty)$, let
\begin{equation}
L_{T}(g )=\int_{\mathbb{X}_{T}}l(g (t,z))\nu _{T}(dt\, dz).
\label{Ltdef}
\end{equation}%
\begin{equation}
S^{n}=\left\{ g:\mathbb{X}_{T}\rightarrow \lbrack 0,\infty ):L_{T}(g)\leq
n\right\} .  \label{eqn:SN}
\end{equation}%
A function $g\in S^{n}$ can be identified with a measure $\nu _{T}^{g}\in
\mathbb{M}$, defined by
\begin{equation*}
\nu _{T}^{g}(A)=\int_{A}g(s,x)\nu _{T}(ds\, dx),\quad A\in \mathcal{B}(\mathbb{X%
}_{T}).
\end{equation*}%
This identification induces a topology on $S^{n}$ under which $S^{n}$ is a
compact space. (See \cite{BudChenDup2013} for a proof.)
Let $S = \cup_{n \ge 1} S^n$.
For $g \in S$ consider the integral equation
\begin{equation}
	\label{eq:eqcontode}
	\xi(t) = \int_{\X_t} h(z,\xi(s)) g(s,z) \nu(dz) ds, \; t \in [0, T].
\end{equation}
The following result says that the above integral equation has a unique solution for every $g \in S$. Proof is given in Section \ref{sec:three}.
\begin{theorem}
	\label{thm:odewell}
For every $g \in S$ there is a unique $\xi \in C([0,T]: \R_+^d)$ that solves the integral equation \eqref{eq:eqcontode}.
\end{theorem}

\subsection{Large Deviation Principle}
\label{sec:two-2}	
We are now ready to present our main result.  Given $g \in S$ denote by $\xi^g$ the unique solution of \eqref{eq:eqcontode}.
Define $I: D([0,T]: \R_+^d) \to [0,\infty]$
as
\begin{equation}
	\label{eq:ratefn}
	I(\phi) = \inf_{g \in S: \phi = \xi^g} \{L_T(g)\},
	\end{equation}
	where infimum over an empty set is taken to be $\infty$.
	In particular this says that $I(\phi) = \infty$ for all $\phi \in D([0,T]: \R_+^d)\setminus C([0,T]: \R_+^d)$.
	The following is our main result.
\begin{theorem}
	\label{thm:main}
	$I$ is a rate function and the collection $\{X^{\eps}\}_{\eps>0}$ satisfies a LDP in $D([0,T]: \R_+^d)$ with rate function $I$, as $\eps \to 0$.
\end{theorem}
\begin{remark}
	The LDP for the scalar state-independent case established in Proposition 3.1 of \cite{GaneshMacciTorrisi2005} is an immediate
	consequence of Theorem \ref{thm:main}.  To see this, note that when $d=1$ and
	$H_{\eps}(t,z,x) \equiv H(t,z)$, where $H$ is as introduced in \eqref{eq:procX}, Condition \ref{cond:uniqsolv} holds trivially.  Furthermore, under the assumptions made in \cite{GaneshMacciTorrisi2005} (specifically, \eqref{eq:eq619}),   parts (a)-(f) of Condition  \ref{cond:mainonh} are immediate and $h(z,x)\equiv h(z)$.  Finally the requirement in 
	\eqref{eq:eq601/26}, and since $\nu$ is a probability measure, implies that part (g) of Condition  \ref{cond:mainonh} holds as well in this state-independent case.  
\end{remark}

\section{Proof of Theorem \ref{thm:odewell}}
\label{sec:three}
In this section we prove Theorem \ref{thm:odewell}.  We start with the following two lemmas which will be used several times in this work.
The proof of the first lemma is standard and is omitted.
\begin{lemma}
\label{sineqs}
\begin{enumerate}[(a)]
\item For $a,b\in (0,\infty )$ and $\sigma \in \lbrack 1,\infty )$, $%
ab\leq e^{\sigma a}+\frac{1}{\sigma }\ell (b).$

\item For every $\beta >0$, there exist $\vr_{1}(\beta ),\vr_2(\beta )\in (0,\infty )$ such that $\vr _{1}(\beta ),\vr_2(\beta )\rightarrow 0$ as $\beta \rightarrow \infty $, and
\begin{equation*}
|x-1|\leq \vr_{1}(\beta )\ell (x)\mbox{ for }|x-1|\geq \beta ,x\geq 0,%
\mbox{
and }\;x\leq \vr_2(\beta )\ell (x)\mbox{ for 
}x\geq \beta > 1.
\end{equation*}
\end{enumerate}
\end{lemma}

\begin{lemma}
	\label{lem:tightest}
	Let $\vartheta \in \Lexp \cap \cll^1(\nu)$.  For every $\delta > 0$ and $n \in \mathbb{N}$, there exists $c(\delta, n,\vartheta) \in (0, \infty)$
	such that for all $\tilde \vartheta: \X \to \mathbb{R}_+$ such that $\tilde \vartheta \le \vartheta$, all measurable maps $f: [0,T] \to \mathbb{R}_+$ and $0 \le s \le t \le T$
	\begin{align}
	 \sup_{g \in S^n} \int_{ (s, t] \times \X} f(u) \tilde \vartheta(z) g(u,z) \nu(dz)\, du
	 \le c(\delta,n, \vartheta) \left(\int_{\X} \tilde \vartheta(z) \nu(dz)\right) \left(\int_s^t f(u) du\right) + \delta |f|_{*,t}.\label{eq:tightbd}
\end{align}
\end{lemma}
\begin{proof}
	Let $f:[0,T] \to \mathbb{R}_+$, $g \in S^n$ and $\tilde \vartheta, \vartheta$ be as in the statement of the lemma.  Then for each $m > 0$
	$$
			\int_{ (s, t]\times \X} f(u) \tilde \vartheta(z) g(u,z) \nu(dz)\, du = T_1(m) + T_2(m),
	$$
	where
	$$
	T_1(m) = \int_{ (s, t] \times \{\vartheta \le m\}} f(u) \tilde \vartheta(z) g(u,z) \nu(dz)\, du,$$
	and
	$$
	T_2(m) = \int_{(s, t] \times \{\vartheta > m\}} f(u) \tilde \vartheta(z) g(u,z) \nu(dz)\, du.$$
	Using Lemma \ref{sineqs}(a), we can estimate $T_2(m)$, for each $k \ge 1$, as
	\begin{align*}
		T_2(m) \le |f|_{*,t} \left( T\int_{\{\vartheta > m\}} e^{k \vartheta(z)} \nu(dz) +  \frac{n}{k}\right).
	\end{align*}
	For each $\beta > 1$, define the sets $E_1(m,\beta)$ and $E_2(m,\beta)$ by
	\begin{align*}
	E_1(m,\beta) &= \{(u,z) \in (s,t] \times  \X: \vartheta(z) \le m \mbox{ and } g(s,z) \le \beta\},\\
	E_2(m,\beta) &= \{(u,z) \in (s,t] \times  \X: \vartheta(z) \le m \mbox{ and } g(s,z) > \beta\}.
	\end{align*}
	Then, $T_1(m)$ can be estimated as 
	$$T_1(m) \le T_3(m,\beta) + T_4(m,\beta),$$
	where
	\begin{align*} 
	T_3(m,\beta) = \int_{E_1(m,\beta)} f(u) \tilde \vartheta(z) g(u,z) \nu(dz)du, \; T_4(m,\beta) = \int_{E_2(m,\beta)} f(u) \tilde \vartheta(z) g(u,z) \nu(dz)du.
\end{align*}
	Using Lemma \ref{sineqs}(b)
	$$T_3(m,\beta) + T_4(m,\beta) \le \beta \left(\int_{\X} \tilde \vartheta(z) \nu(dz)\right) \left(\int_s^t f(u) du\right) + \vr_2(\beta) m n|f|_{*,t}.$$
	Combining the estimates for $T_1(m)$ and $T_2(m)$, the left side of \eqref{eq:tightbd} can be bounded by 
	\begin{align*}
	&  \beta \left(\int_{\X} \tilde \vartheta(z) \nu(dz)\right) \left(\int_s^t f(u) du\right)\\
		&\quad + |f|_{*,t} \left( \vr_2(\beta)
		 m n + T\int_{\{\vartheta > m\}} e^{k \vartheta(z)} \nu(dz) +  \frac{n}{k}\right).
	\end{align*}
	Now given $\delta > 0$, choose $k > 1$ such that $\frac{n}{k} < \frac{\delta}{3}$.
	Next, since $\vartheta \in \Lexp \cap \cll^1(\nu)$, it is possible to choose $m > 0$ such that $T \int_{\{\vartheta > m\}} e^{k \vartheta(z)} \nu(dz) < \frac{\delta}{3}$.
	Finally using Lemma \ref{sineqs}(b) choose $\beta > 1$ such that
	$\vr_2(\beta)
	 m n < \frac{\delta}{3}$.  The result now follows on taking $c(\delta, n,\vartheta) = \beta$.
\end{proof}

{\bf Proof of Theorem \ref{thm:odewell}.}
We will use Banach's fixed point theorem. Fix $n \in \N$ and $g \in S^n$.  Define for $r> 0$,
$T^r: C([0,r]: \R_+^d) \to C([0,r]: \R_+^d)$ as
$$
T^r(\phi)(t) = \phi(0)+ \int_{\X_t} h(z,\phi(s)) g(s,z) \nu(dz) ds, \; t \in [0,r], \, \phi \in C([0,r]: \R_+^d).$$
Note that the right-hand side indeed defines an element of $C([0,r]: \R_+^d)$ since by
Lemma \ref{lem:tightest}, for $\delta > 0$ and $0 \le s < t < r$, 
$$
\int_{(s,t]\times \X} \|h(z,\phi(u))\| g(s,u) \nu(dz) du \le (1 + \|\phi\|_{*,r})\left(c(\delta, n, M_h) (t-s) \int_{\X} M_h(z) \nu(dz) + \delta\right).$$
We will now argue that for $r$ small enough $T^r$ is a contraction.
Note that for $\phi, \ti \phi \in C([0,r]: \R_+^d)$, with $\phi(0) = \tilde \phi(0)$,
\begin{align*}
	\|T^r(\phi) - T^r(\ti \phi)\|_{*,r} & \le \int_{\X_r} \|h(z,\phi(s)) - h(z,\ti \phi(s))\|g(s,z) \nu(dz) ds\\
	&\le \|\phi - \ti \phi\|_{*,r} \int_{\X_r} L_h(z) g(s,z) \nu(dz) ds.
\end{align*}
Using Lemma \ref{lem:tightest} again and our assumption on $L_h$ we have that  $\int_{\X_T} L_h(z) g(s,z) \nu(dz) ds < \infty$. 
Thus  for $r$ sufficiently small $\int_{\X_r} L_h(z) g(s,z) \nu(dz) ds < 1$
and consequently $T^r$ is a contraction and so by Banach's fixed point theorem has a unique fixed point.
This shows that there is a unique solution to \eqref{eq:eqcontode} for all $t \in [0,r]$.  
The result now follows by a recursive argument.
\qed 

\section{A General Sufficient Condition for LDP.}
\label{sec:four}
We now present a result from \cite{BudDupMar2011} which will be a key ingredient in our proofs.  We begin with some notation.
Let $\bar{\mathcal{A}}$ be the class of all $(\bar{\mathcal{P}}\otimes
\mathcal{B}(\mathbb{X}))/\mathcal{B}[0,\infty )$-measurable maps $\varphi :%
\mathbb{X}_{T}\times \bar{\mathbb{M}}\rightarrow \lbrack 0,\infty )$; as is common we frequently suppress in the notation the dependence of $\varphi$ on elements in (the probability space) $\bar \bM$. For $%
\varphi \in \bar{\mathcal{A}}$, define a counting process $N^{\varphi }$ on $%
\mathbb{X}_{T}$ by
\begin{equation}
N^{\varphi }((0,t]\times U)=\int_{(0,t]\times U\times \R_+}1_{[0,\varphi (s,z)]}(r)\bar{N}(ds\, dz\, dr),\quad t\in \lbrack 0,T], \ U\in
\mathcal{B}(\mathbb{X}).  \label{Eqn:control}
\end{equation}%
$N^{\varphi }$ can be interpreted as a controlled random measure, with $\varphi$
playing the role of the control which selects the intensity for the points at location $x$ and time $s$, in a
possibly random but non-anticipating way.
Let
\begin{equation*}
\clu ^{n}=\{\varphi \in \bar{\mathcal{A}}:(s,z)\mapsto \varphi(s,z,\om) \in S^{n}, \ \bar{%
\mathbb{P}}\ a.e.\ \om\}.
\end{equation*}
Elements of $\clu ^n$ will be regarded as $S^n$-valued random variables where the topology on the latter  space is as introduced below \eqref{eqn:SN}.
Let $\left\{ K_{m}\subset \mathbb{X},m=1,2,\ldots \right\} $ be an
increasing sequence of compact sets such that $\cup _{m=1}^{\infty }K_{m}=%
\mathbb{X}$. For each $m$ let
\begin{eqnarray*}
\bar{\mathcal{A}}_{b,m}&=& \left\{ \varphi \in \bar{\mathcal{A}}: \mbox{ for all }(t,\omega )\in [0,T]\times \mathbb{\bar{M}}\mbox{, }m\geq \varphi
(t,x,\omega )\geq 1/m\mbox{ if }x\in K_{m}\right. \\
&& \hspace{1in} \left. \mbox{ and }\varphi (t,x,\omega )=1\mbox{ if }x\in
K_{m}^{c}\right\} ,
\end{eqnarray*}%
and let $\bar{\mathcal{A}}_{b}=\cup_{m=1}^{\infty }\bar{\mathcal{A}}_{b,m}$. Define $\tilde{\mathcal{U}}^n=\mathcal{U}^n\cap \bar{\mathcal{A}}_{b}$.

Let $\U$  be a Polish space. The following condition is a slight modification of a condition introduced in Section 4 in \cite{BudDupMar2011} to establish a large deviation result; see Section 2.2 in \cite{BudChenDup2013}.
\begin{condition}
\label{cond:gensuff} There exist measurable maps $\mathcal{G}^0$, $\clg^{\eps}$, $\eps > 0$  from $\mathbb{M}$
to $\mathbb{U}$ such that the following hold.
\begin{enumerate}[(a)]
\item For $n \in \mathbb{N}$, let $g_m, g \in S^n$ be such that $g_m
\rightarrow g$ as $m \rightarrow \infty$. Then
\begin{equation*}
\mathcal{G}^0\left(\nu_T^{g_m}\right)\rightarrow\mathcal{G}%
^0\left(\nu_T^{g}\right).
\end{equation*}

\item For $n\in \mathbb{N}$, let $\varphi _{\eps },\varphi \in \tilde{\mathcal{U}}^{n}$ be such that $\varphi_{\eps }$ converges in distribution to $\varphi $ as $\eps \rightarrow 0$. Then
\begin{equation*}
\mathcal{G}^{\eps }(\eps N^{\eps ^{-1}\varphi _{\eps
}})\Rightarrow \mathcal{G}^{0}\left( \nu _{T}^{\varphi }\right) .
\end{equation*}
\end{enumerate}
\end{condition}

For $\phi \in \mathbb{U}$, define $\mathbb{S}_{\phi }=\left\{ g\in S%
:\phi =\mathcal{G}^{0}(\nu _{T}^{g})\right\} $. Let $I:\mathbb{U}\rightarrow
\lbrack 0,\infty ]$ be defined by
\begin{equation}
I(\phi )=\inf_{g\in \mathbb{S}_{\phi }}\left\{ L_{T}(g)\right\} ,\quad \phi
\in \mathbb{U}.  \label{Eqn: I2}
\end{equation}%
By convention, $I(\phi )=\infty $ if $\mathbb{S}_{\phi }=\emptyset$.

The following theorem is a slight extension of Theorem 4.2 of \cite{BudDupMar2011}.
For a proof, we refer the reader to the Appendix of \cite{BudChenDup2013}.
\begin{theorem}
\label{Thm:LDP01} For $\eps >0$, let $Z^{\eps }$ be defined by $%
Z^{\eps }=\mathcal{G}^{\eps }(\eps N^{\eps ^{-1}})$.
If condition \ref{cond:gensuff} holds, then $I$ defined as in (\ref{Eqn: I2}) is a rate function on $\mathbb{U}$ and the family $\{Z^{\eps
}\}_{\eps >0}$ satisfies a large deviation principle with rate function $%
I$.
\end{theorem}

\section{Proof of Theorem \ref{thm:main}}
\label{sec:five}
In order to prove Theorem \ref{thm:main} we will apply Theorem \ref{Thm:LDP01} with 
$\U = D([0,T]: \R_+^d)$,
$\clg^{\eps}$ as introduced in Condition \ref{cond:uniqsolv},
and $\clg^0: \MM \to C([0,T]: \R_+^d)$ defined as follows. Let
$\clg^0(m) = \xi^g$ if $m = \nu^g_T$ for some $g \in S$ where $\xi^g$ is as introduced above
\eqref{eq:ratefn}.  For all other $m \in \MM$ we set $\clg^0(m) = 0$. It suffices to show that Condition \ref{cond:gensuff} is satisfied with this choice of $\clg^{\eps}$ and $\clg^0$.
In Section \ref{sec:five-1} we will verify part (a) of this condition and Section \ref{sec:five-2} is devoted to the verification of part (b).
\subsection{Verification of Condition \ref{cond:gensuff}(a)}
\label{sec:five-1}
The following is the main result of this section.
\begin{proposition}
	\label{prop:parta}
	Let $n \in \N$ and $g_k, g \in S^n$, $k \ge 1$ be such that $g_k \to g$.  Then $\xi^{g_k} \to \xi^g$ in
	$C([0,T]: \R_+^d)$.
	\end{proposition}
	The following lemma will be useful in proving the proposition.
	\begin{lemma}
		\label{lem:cgce1050}
		Let $f_k, f \in D([0,T]:\R_+^d)$, $k \ge 1$, be such that $\|f_k-f\|_{*,T} \to 0$ as $k\to \infty$.  Also let $n \in \N$ and
		 $g_k, g \in S^n$ be such that $g_k\to g$.  Then, letting
		$$\ti f_k(t) = 	\int_{\X_t} h(z,f_k(s)) g_k(s,z) \nu(dz) ds, \; \ti f(t) = 	\int_{\X_t} h(z,f(s)) g(s,z) \nu(dz) ds,$$
		$\ti f_k(t) \to \ti f(t)$ as $k \to \infty$, for every $t \in [0,T]$.
		\end{lemma}
	\begin{proof}
	 Note that
		\begin{align}
			\int_{\X_t} h(z,f_k(s)) g_k(s,z) \nu(dz) ds - \int_{\X_t} h(z,f(s)) g(s,z) \nu(dz) ds
			= T_1^k(t) + T_2^k(t),\label{eq:eq839/27}
			\end{align}
			where
			$$T_1^k(t) = \int_{\X_t} \left[h(z,f_k(s)) - h(z, f(s))\right]  g_k(s,z) \nu(dz) ds$$
			and
			$$T_2^k(t) = \int_{\X_t}  h(z, f(s)) \left[g_k(s,z) - g(s,z)\right] \nu(dz) ds.$$
			Noting that
			\begin{align*}
				\|T_1^k\|_{*,T} \le \|f_k - f\|_{*,T} \int_{\X_T}  L_h(z) g_k(s,z) \nu(dz) ds
			\end{align*}
			and $\sup_k \int_{\X_T}  L_h(z) g_k(s,z) \nu(dz) ds < \infty$ from Lemma \ref{lem:tightest},
			we see that $\|T_1^k\|_{*,T} \to 0$ as $k\to \infty$.
			Consider now $T_2^k$.  
			We first claim that for every $\epsilon > 0$ there is a compact $K \subset \X$ such that
			\begin{equation}
				\label{eq:eq720/27}
				\sup_{\psi \in S^n} \int_{[0,T]\times K^c} M_h(z) \psi(s,z) \nu(dz) ds < \epsilon .
						\end{equation}
						To see the claim, let $\{K_{\gamma}\}_{\gamma \in \N}$ be a sequence of compact subsets of $\X$ such that $K_{\gamma}\uparrow \X$
						as $\gamma \to \infty$.  Since, $M_h \in \cll^1(\nu)$
						$$\int_{K_{\gamma}^c} M_h(z) \nu(dz) \to 0 \mbox{ as } \gamma \to \infty .$$
						Also, from Lemma  \ref{lem:tightest}, with $f(u) \equiv 1$, $\tilde \vartheta(z) = 1_{K_{\gamma}^c}(z) M_h(z)$ and
						$\vartheta(z) = M_h(z)$ we have that for every $\delta > 0$
						$$\sup_{\psi \in S^n} \int_{[0,T]\times K_{\gamma}^c} M_h(z) \psi(s,z) \nu(dz) ds
						\le c(\delta,n, M_h) T \int_{K_{\gamma}^c} M_h(z) \nu(dz) + \delta .$$
						The claim now follows on combining the above two displays.
						Using \eqref{eq:eq720/27}, for a fixed $\epsilon > 0$, choose a compact $K \subset \X$ such that
						$$T_2^k(t) = \int_{\X_t}  h(z, f(s))1_{K}(z) \left[g_k(s,z) - g(s,z)\right] \nu(dz) ds + T_3^k(t)$$
						and $\sup_{k\ge 1}\|T_3^k\|_{*,T} \le \epsilon$.
						Next, for $\rho >0$ write
						$$\int_{\X_t}  h(z, f(s))1_{K}(z) \left[g_k(s,z) - g(s,z)\right] \nu(dz) ds = T_{4,\rho}^k(t) + T_{5,\rho}^k(t)$$
						where
						$$T_{4,\rho}^k(t) = \int_{[0,t]\times K}  h(z, f(s))1_{M_h(z) \le \rho} \left[g_k(s,z) - g(s,z)\right] 
						\nu(dz) ds$$
						and
						$$T_{5,\rho}^k(t) = \int_{[0,t]\times K}  h(z, f(s))1_{M_h(z) > \rho} \left[g_k(s,z) - g(s,z)\right] 
						\nu(dz) ds$$
							From Lemma \ref{lem:tightest}, for every $\delta > 0$,
						$$\sup_k\|T_{5,\rho}^k\|_{*,T} \le 
						(1 + \|f\|_{*,T})\left(2c(\delta,n, M_h) T \int_{\X} M_h(z)1_{M_h(z)> \rho} \nu(dz) + 2\delta\right),$$
						Choose $\delta > 0$ and $\rho > 0$ such that the right-hand side of the above expression is bounded by $\epsilon$.
						A minor modification of Lemma 2.8 of \cite{BoDu} shows (see Appendix A.6 of \cite{BudChenDup2013})  that
						for every $\rho > 0$, $T_{4,\rho}^k(t) \to 0$
						as $k \to \infty$.  Combining the above estimates we have that for every $t \in [0,T]$, $\limsup_{k\to \infty}\|T_2^k(t)\| \le 2\epsilon$.
				Since $\epsilon > 0$ is arbitrary, the above implies that for every $t \in [0,T]$, $T_2^k(t) \to 0$ as $k\to \infty$.
						Thus we have shown that the expression on the left-hand side of \eqref{eq:eq839/27} converges to $0$ as $k \to \infty$,
						which proves the result.
	\end{proof}

	{\em Proof of Proposition \ref{prop:parta}.}
		Let $\xi_k = \xi^{g_k}$, $\xi = \xi^g$.  We first argue that $\{\xi_k\}_{k\ge 1}$ is pre-compact in $C([0,T]: \R_+^d)$. Note that
		\begin{align*}
			\|\xi_k(t)\| \le \int_{\X_t} (1 + \|\xi_k(s)\|)M_h(z) g_k(s,z) \nu(dz) ds, \; t \in [0, T].
			\end{align*}
		From Lemma \ref{lem:tightest} it follows that for any $\delta > 0$
		\begin{equation}\label{eq:1129/28}
		\sup_{\psi \in S^n} \int_{\X_T} M_h(z) \psi(s,z) \nu(dz) ds \le c(\delta,n, M_h)T \int_{\X}M_h(z) \nu(dz) + \delta.
	\end{equation}
		An application of Gronwall's lemma now shows that
	\begin{equation}\label{eq:eq638/27}\sup_{k\ge 1}(1+\|\xi_k\|_{*,T})  = \kappa_1 < \infty.\end{equation}
		Next, for $0 \le s \le t \le T$ and $\delta > 0$.
		\begin{align*}
			\|\xi_k(t)-\xi_k(s)\| &\le \int_{(s,t]\times \X} \|h(z,\xi_k(u))\| g_k(u,z) \nu(dz) du\\
			&\le \kappa_1 \int_{(s,t]\times \X} M_h(z) g_k(u,z) \nu(dz) du\\
			&\le \kappa_1\left(c(\delta,n, M_h)(t-s) \int_{\X}M_h(z) \nu(dz) + \delta\right).
		\end{align*}
		This shows that $\{\xi_k\}_{k\ge 1}$ is equicontinuous which together with \eqref{eq:eq638/27} proves the desired pre-compactness. 
		Suppose $\xi_k$ converges along a subsequence to $\bar\xi$.  
		From Lemma \ref{lem:cgce1050}, along this subsequence, for every $t \in [0,T]$, as $k \to \infty$, 
		$$\int_{\X_t} h(z,\xi_k(s)) g_k(s,z) \nu(dz) ds \to \int_{\X_t} h(z,\bar \xi(s)) g(s,z) \nu(dz) ds$$
						Combining this with the fact that $\xi_k$ solves 
							$$\xi_k(t) = \int_{\X_t} h(z,\xi_k(s)) g_k(s,z) \nu(dz) ds, \; t \in [0, T]$$
							for every $k \ge 1$ and that $\xi_k$ converges along the chosen subsequence to $\bar \xi$ we have that
							$$\bar \xi(t) = \int_{\X_t} h(z,\bar \xi(s)) g(s,z) \nu(dz) ds, \; t \in [0, T].$$
							By unique solvability of the above equation and the definition of $\xi$ we now see that $\bar \xi = \xi$.
							\qed

\subsection{Verification of Condition \ref{cond:gensuff}(b)}
\label{sec:five-2}
The following is the main result of this section.
\begin{proposition}
	\label{prop:verifyb}
	Let  $n\in \mathbb{N}$ and let $\varphi_{\eps},\varphi \in \tilde{\mathcal{U}}^{n}$ be such that $\varphi_{\eps}$ converges in distribution to $\varphi $ as $\eps \rightarrow 0$.  Let $\{\clg^{\eps}\}_{\eps > 0}$ be as in 
	Condition \ref{cond:uniqsolv} and $\clg^0$ be as introduced at the beginning of Section \ref{sec:five}.
	Then $
	\mathcal{G}^{\eps}(\eps N^{\eps^{-1}\varphi _{\eps}})\Rightarrow \mathcal{G}^{0}\left( \nu _{T}^{\varphi }\right) .
	$
\end{proposition}
\begin{proof}
	Let $\tilde{\varphi}_{\eps}=1/\varphi_{\eps}$, and
	recall that $\varphi_{\eps}\in \tilde{\mathcal{U}}^{n}$ means that $\varphi_{\eps}=1$
	off some compact set in $\XX$ and bounded above and below away from zero on the
	compact set. Then it is easy to check (see Theorem III.3.24 of \cite{JacodShiryaev},
	see also Lemma 2.3 of \cite{BudDupMar2011}) that
	\[
	\mathcal{E}_{t}^{\varepsilon} =\exp\left\{  \int
	_{\mathbb{(}0,t]\times\X \times\lbrack0,\varepsilon^{-1}\varphi_{\eps}]}%
	\log(\tilde{\varphi_{\eps}})d\bar{N}+\int_{\mathbb{(}0,t]\times\X%
	\times\lbrack0,\varepsilon^{-1}\varphi_{\eps}]}\left(  -\tilde{\varphi}_{\eps}+1\right)
	d\bar{\nu}_{T}\right\}
	\]
	is an $\left\{  \bar{\mathcal{F}}_{t}\right\}  $-martingale and consequently
	\[
	\mathbb{Q}_{T}^{\varepsilon}(G)=\int_{G}\mathcal{E}_{T}^{\varepsilon}%
	(\tilde{\varphi})d\bar{\PP},\quad\text{ for }G\in\mathcal{B(}%
	\mathbb{\bar{M}}\mathcal{\mathbb{\mathcal{)}}},%
	\]
	defines a probability measure on $\mathbb{\bar{M}}$. Furthermore,
	$\bar{\PP}$ and $\mathbb{Q}_{T}^{\varepsilon}$ are mutually absolutely
	continuous and it can be verified that under $\mathbb{Q}_{T}^{\varepsilon}
	$, $\varepsilon N^{\varepsilon^{-1}\varphi _{\varepsilon}}$ has the same law as that of
	$\varepsilon N^{\varepsilon^{-1}}$ under $\bar{\PP}$. Thus from Condition \ref{cond:uniqsolv} it follows
	that $\bar{X}^{\varepsilon}={\mathcal{G}}^{\varepsilon
	}(\varepsilon N^{\varepsilon^{-1}\varphi_{\eps}})$ is $\mathbb{Q}_{T}^{\varepsilon}$
	a.s.\ (and hence $\bar{\PP}$ a.s.) the unique solution of
	\begin{equation}
		\label{eq:eq627/26cont}
		\bar X^{\eps}(t) = \eps\int_{\X_t} H_{\eps}(\eps^{-1}(t-s), z, \eps^{-1}\bar X^{\eps}(s-)) N^{\eps^{-1}\varphi_{\eps}}(ds\, dz),\; t \in [0, T].
	\end{equation}
Also note that $\bar X^0 = \clg^0\left( \nu _{T}^{\varphi }\right)$	solves the integral equation
\begin{equation}\label{eq:eq647}
	\bar X^0(t) = \int_{\X_t} h(z,\bar X^0(s)) \varphi(s,z) \nu(dz) ds, \; t \in [0, T].
\end{equation}
In order to prove the result we need to show that $\bar X^{\eps}$ converges in distribution to $\bar X^0$.  We start by showing that $\{\bar X^{\eps}\} _{\eps > 0}$ is tight.
Note that, for $t \in [0,T]$,
\begin{align*}
\bar \E\|\bar X^{\eps}\|_{*,t} &\le \bar \E\int_{\X_t} \|\bar H(\eps^{-1}(t-s), z, \bar X^{\eps}(s))\| \varphi_{\eps}(s,z) \nu(dz) ds \\
&\quad + 
\bar \E\int_{\X_t} \|R_{\eps}(\eps^{-1}(t-s), z, \bar X^{\eps}(s))\| \varphi_{\eps}(s,z) \nu(dz) ds\\
&= T_1^{\eps}(t) + T_2^{\eps}(t).
\end{align*}
Using the monotonicity of $\bar H$ we see that
\begin{align*}
	T_1^{\eps}(t) &\le \bar \E \int_{\X_t} \|h(z, \bar X^{\eps}(s))\| \varphi_{\eps}(s,z) \nu(dz) ds\\
	&\le \bar \E \int_{\X_t} ( 1 +  \|\bar X^{\eps}\|_{*,s}) M_h(z)\varphi_{\eps}(s,z) \nu(dz) ds.
\end{align*}
	Using Lemma \ref{lem:tightest} we now have that for every $\delta > 0$
	\begin{equation}
		\label{eq:eq1034/28}
		T_1^{\eps}(t) \le c(\delta, n, M_h) (\int_{\X} M_h(z) \nu(dz)) \int_0^t ( 1+ \bar \E \|\bar X^{\eps}\|_{*,s}) ds + \delta ( 1+ \bar \E\|\bar X^{\eps}\|_{*,t}).
	\end{equation}
Another application of Lemma \ref{lem:tightest} shows that for each fixed $\delta > 0$
\begin{align*}
	T_2^{\eps}(t) &\le \bar \E \int_{\X_t} \vs_{\eps}(z) (1+ \bar \E \|\bar X^{\eps}\|_{*,s})  \varphi_{\eps}(s,z) \nu(dz) ds\\
	&\le c(\delta, n, \vs) \int_{\X} \vs_{\eps}(z) \nu(dz)\int_0^t ( 1+ \bar \E \|\bar X^{\eps}\|_{*,s}) ds + \delta ( 1+ \bar \E\|\bar X^{\eps}\|_{*,t})
\end{align*}
Combining the above estimates on $T_1^{\eps}$ and $T_2^{\eps}$	and choosing $\delta$ sufficiently small, we have by an application of Gronwall's lemma that
\begin{equation}
	\label{eq:1131/28}
	\sup_{\eps > 0} \bar \E\|\bar X^{\eps}\|_{*,T} = \kappa_1 < \infty.
\end{equation}
In order to prove the tightness of $\{\bar X^{\eps}\}_{\eps > 0}$ we will first establish the tightness of  the following closely related collection $\{\ti X^{\eps}\}_{\eps > 0}$
of
$C([0, T]: \R^d)$-valued random variables:
\begin{equation}
	\label{eq:1115/29}
\ti X^{\eps}(t) = \int_{\X_t} h(z,\bar X^{\eps}(s)) \varphi_{\eps}(s,z) \nu(dz) ds, \, t \in [0,T], \, \eps > 0.
\end{equation}
For that we first observe that
$$\|\ti X^{\eps}\|_{*,T} \le (1+ \|\bar X^{\eps}\|_{*,T}) \int_{\X_T} M_h(z) \varphi_{\eps}(s,z) \nu(dz) ds.$$
Combining the above estimate with \eqref{eq:1129/28} and \eqref{eq:1131/28} we now see that
$\sup_{\eps > 0} \bar \E\|\ti X^{\eps}\|_{*,T} < \infty$.  Also, for $0 \le s \le t \le T$ and $\delta > 0$.
\begin{align*}
	\|\ti X^{\eps}(t)-\ti X^{\eps}(s)\| &\le \int_{(s,t]\times \X} \|h(z,\bar X^{\eps}(u))\| \varphi_{\eps}(u,z) \nu(dz) du\\
	&\le (1+\|\bar X^{\eps}\|_{*,T}) \left(c(\delta,n, M_h)(t-s) \int_{\X}M_h(z) \nu(dz) + \delta\right).
\end{align*}
Let $\kappa_2(\delta) = c(\delta,n, M_h)\int_{\X}M_h(z) \nu(dz)$ and consider
$$A_{\alpha} = \{x \in C([0,T]:\R^d): \|x\|_{*,T}\le \alpha, \mbox{ and for every } \delta > 0, \|x(t) - x(s)\| \le \alpha(\kappa_2(\delta)(t-s) + \delta)\}.$$
It is easy to check that for every $\alpha > 0$, $A_{\alpha}$ is a compact subset of $C([0,T]:\R^d)$.
Also from the above estimates, $\sup_{\eps}\bar \P(\ti X^{\eps} \in A_{\alpha}^c) \to 0$ as $\alpha \to \infty$.
This proves the tightness of $\{\ti X^{\eps}\}_{\eps > 0}$.
Next, let for $\eps > 0$,
$$\bar Y^{\eps}(t) = \eps \int_{\X_t} h(z, \bar X^{\eps}(s-)) N^{\eps^{-1}\varphi_{\eps}}(ds\, dz), \, t \in [0,T].$$
Then, for $t \in [0,T]$,
\begin{equation}
	\label{eq:1248/28}
	\bar Y^{\eps}(t) - 	\bar X^{\eps}(t) = \eps \int_{\X_t} \left[h(z, \bar X^{\eps}(s-)) - 
	\bar H(\eps^{-1}(t-s), z, \bar X^{\eps}(s-))\right]  N^{\eps^{-1}\varphi_{\eps}}(ds\, dz)
	+ \clr_1^{\eps}(t),
\end{equation}
where
\begin{align*}
	\|\clr_1^{\eps}\|_{*,T} &\le \eps \sup_{t \in [0,T]} \int_{\X_t} \|R_{\eps}(\eps^{-1}(t-s), z, \bar X^{\eps}(s-))\|
	N^{\eps^{-1}\varphi_{\eps}}(ds\, dz)\\
	&\le  \eps\int_{\X_T}  \vs_{\eps}(z) (\|\bar X^{\eps}(s-))\| + 1) 
	N^{\eps^{-1}\varphi_{\eps}}(ds\, dz).
\end{align*}	
Thus, for every $\delta > 0$
\begin{align*}
	\bar \E \|\clr_1^{\eps}\|_{*,T} &\le \bar \E\left( (\|\bar X^{\eps}\|_{*,T} + 1) \int_{\X_T} \vs_{\eps}(z)\varphi_{\eps}(s,z)
	\nu(dz) ds\right).\\
	&\le    (\bar\E\left\|\bar X^{\eps}\|_{*,T} + 1\right) \left[c(\delta, n, \vs)T \int_{\X} \vs_{\eps}(z)\nu(dz) +
	\delta\right].
\end{align*}
	Since $\int_{\X} \vs_{\eps}(z)\nu(dz)$ converges to $0$ as $\eps \to 0$, we have that
	\begin{equation}
		\label{eq:1027/29}
	\clr_1^{\eps} \mbox{ converges to } 0 \mbox{ in probability in } D([0,T]:\R^d).
\end{equation}
	Next, denoting the first term on the right side of \eqref{eq:1248/28} as 
	$\cls_1^{\eps}(t)$, we  have for $t_0 \in [0, T]$,
	\begin{align*}
		\bar \E\|\cls_1^{\eps}\|_{*,t_0} & \le \eps \bar \E \int_{\X_{t_0}} \|h(z, \bar X^{\eps}(s-))\|  N^{\eps^{-1}\varphi_{\eps}}(ds\, dz)\\
		&\le \bar \E \left((\|\bar X^{\eps}\|_{*,T}+ 1) \int_{\X_{t_0}} M_h(z)\varphi_{\eps}(s,z) \nu(dz) ds\right)\\
		&\le \left( \bar \E \|\bar X^{\eps}\|_{*,T} + 1\right) \left[c(\delta, n, M_h)t_0 \int_{\X} M_h(z)\nu(dz) +
		\delta\right].
	\end{align*}
	Thus, for some $\kappa_1 \in (0, \infty)$, we have for every $\delta > 0$,
	\begin{equation}
		\sup_{\eps > 0} \bar \E\|\cls_1^{\eps}\|_{*,t_0} \le \kappa_1(t_0c(\delta, n, M_h) + \delta).
		\label{eq:1023/29}
		\end{equation}
		Now we consider the interval $(t_0, T]$. Note that, for any $\ups \in (0, t_0)$
		\begin{equation}\label{eq:1017/29}
			\sup_{t \in (t_0, T]}\|\cls_1^{\eps}(t)\| \le 
		\sup_{t \in (t_0, T]}\|\cls_1^{\eps}(t-\ups)\| + \sup_{t \in (t_0, T]} \|\cls_{2,\ups}^{\eps}(t)\|,\end{equation}
		where
		$$
		\cls_{2,\ups}^{\eps}(t) = \eps \int_{(t-\ups, t]\times \X} \left[h(z, \bar X^{\eps}(s-)) - 
		\bar H(\eps^{-1}(t-s), z, \bar X^{\eps}(s-))\right]  N^{\eps^{-1}\varphi_{\eps}}(ds\, dz).$$
		Using the monotonicity of $\bar H$ again, we have that for  $\alpha > 0$,
		\begin{align}
						\sup_{t \in (t_0, T]}\|\cls_1^{\eps}(t-\ups)\| &\le \eps \int_{\X_T} \|h(z, \bar X^{\eps}(s-)) - 
						\bar H(\eps^{-1}\ups, z, \bar X^{\eps}(s-))\| N^{\eps^{-1}\varphi_{\eps}}(ds\, dz)\nonumber\\
						&= \clr_{2,\alpha}^{\eps} + \clr_{3,\alpha}^{\eps},\label{eq:851/28}
										\end{align}
	where
	$$\clr_{2,\alpha}^{\eps} = \eps 1_{B_{\alpha}^{\eps}} \int_{\X_T} \|h(z, \bar X^{\eps}(s-)) - 
	\bar H(\eps^{-1}\ups, z, \bar X^{\eps}(s-))\| N^{\eps^{-1}\varphi_{\eps}}(ds\, dz),$$
	$$\clr_{3,\alpha}^{\eps} =  \eps 1_{(B_{\alpha}^{\eps})^c} \int_{\X_T} \|h(z, \bar X^{\eps}(s-)) - 
	\bar H(\eps^{-1}\ups, z, \bar X^{\eps}(s-))\| N^{\eps^{-1}\varphi_{\eps}}(ds\, dz)$$
	and $B_{\alpha}^{\eps}= \{\omega: \|\bar X^{\eps}\|_{*,T} \le \alpha\}$.
	Let for $\alpha > 0$, $\varpi_{\alpha}: \R_+ \times \XX \to \R_+$ be defined as
	$$
	\varpi_{\alpha}(r,z) = \sup_{\|x\| \le \alpha} \|h(z,x) - \bar H(r,z,x)\|, \; (r,z) \in \R_+\times \X.$$
	Then, from Condition \ref{cond:mainonh}(e), for all $(\alpha, z) \in \R_+\times \X$, 
	$\varpi_{\alpha}(r,z) \to 0$ as $r \to \infty$.  Also, since $\varpi_{\alpha}(r,z) \le M_h(z)(1+\alpha)$ and
	$M_h \in \cll^1(\nu)$, we have that $\int_{\X}\varpi_{\alpha}(r,z) \nu(dz) \to 0$ as $r \to \infty$. Now, for every $\delta > 0,$
	\begin{align*}
		\bar \E \clr_{2,\alpha}^{\eps} &\le \bar \E \int_{\X_T} \varpi_{\alpha}(\eps^{-1}\ups,z) \varphi_{\eps}(s,z) \nu(dz) ds\\
		&\le Tc(\delta, n, M_h)\int_{\X} \varpi_{\alpha}(\eps^{-1}\ups,z) \nu(dz) + \delta (1+\alpha).
			\end{align*}
	Thus $\clr_{2,\alpha}^{\eps} \to 0$ as $\eps \to 0$ for every $\alpha > 0$.
	
	Next, from Markov's inequality and \eqref{eq:1131/28}, for $\eta > 0$
\begin{equation}
	\label{eq:eq558}
\sup_{\eps >0}\bar \P(\clr_{3,\alpha}^{\eps} > \eta) \le \sup_{\eps >0} \bar \P((B_{\alpha}^{\eps})^c)
\le \sup_{\eps >0}\bar \P(\|\bar X^{\eps}\|_{*,T}> \alpha) \le \frac{\kappa_1}{\alpha}.
\end{equation}
	Using the above two observations in \eqref{eq:851/28} we have that for every $t_0 \in (0,T)$ and $\ups \in (0, t_0)$
	\begin{equation}
		\label{eq:1020/29}
	\sup_{t \in (t_0, T]}\|\cls_1^{\eps}(t-\ups)\| \to 0, \mbox{ in probability, as } \eps \to 0.\end{equation}
	We now consider $\cls_{2,\ups}^{\eps}$.  Using monotonicity of $H$,
	\begin{align}
		\sup_{t \in (t_0, T]} \|\cls_{2,\ups}^{\eps}(t)\| &\le \sup_{t \in (t_0, T]} \eps \left\|\int_{(t-\ups, t]\times \X}
		h(z, \bar X^{\eps}(s-)) N^{\eps^{-1}\varphi_{\eps}}(ds\, dz)\right\|\nonumber\\
		&\le \sup_{t \in (t_0, T]} \eps \left\|\int_{(t-\ups, t]\times \X}
		h(z, \bar X^{\eps}(s-)) \ti N^{\eps^{-1}\varphi_{\eps}}(ds\, dz)\right\|\nonumber\\
		&\quad +
		\sup_{t \in (t_0, T]}  \left\|\int_{(t-\ups, t]\times \X}
		h(z, \bar X^{\eps}(s-)) \varphi_{\eps}(s,z) \nu(dz) ds\right\|\nonumber\\
		&= \clr_{4}^{\eps}(t_0, \ups) + \clr_{5}^{\eps}(t_0, \ups),\label{eq:1021/29}
		\end{align}
	where $\ti N^{\eps^{-1}\varphi_{\eps}}(ds\, dz) = N^{\eps^{-1}\varphi_{\eps}}(ds\, dz) - \eps^{-1} \varphi_{\eps}(s,z)\nu_T(ds\, dz)$, the compensated version of $N^{\eps ^{-1} \varphi _{\eps}}$.
	For $\eta > 0$ and $\alpha >0$,
	$$\bar \P(\clr_{5}^{\eps}(t_0, \ups) > \eta) \le \bar \P(\clr_{5}^{\eps}(t_0, \ups) > \eta;\, B_{\alpha}^{\eps}) + \bar \P((B_{\alpha}^{\eps})^c).$$
	Also,
	\begin{align*}
		\clr_{5}^{\eps}(t_0, \ups)1_{B_{\alpha}^{\eps}}
		&\le (1+\alpha) \sup_{t \in (t_0, T]}  \int_{(t-\ups, t]\times \X} M_h(z)\varphi_{\eps}(s,z) \nu(dz) ds\\
		& \le (1+\alpha)\left[\ups c(\delta,n,M_h)\int_{\X}M_h(z)\nu(dz) + \delta\right].
		\end{align*}
		Combining the above two estimates and using \eqref{eq:eq558} once again, we have for every $t_0 \in (0,T)$ and 
		$\eta > 0$,
		\begin{equation}\label{eq:949/29}\sup_{\eps > 0}\bar \P(\clr_{5}^{\eps}(t_0, \ups) > \eta) \to 0, \mbox{ as } \ups \to 0.
		\end{equation}
		We now consider $\clr_{4}^{\eps}(t_0, \ups)$.
		\begin{align}
						\clr_{4}^{\eps}(t_0, \ups) & \le 2\eps\sup_{0\le t \le T} \left\|\int_{\X_t}
						h(z, \bar X^{\eps}(s-)) \ti N^{\eps^{-1}\varphi_{\eps}}(ds\, dz)\right\| = \ti\clr_{4}^{\eps}.
						\label{eq:1034/29}
														\end{align}
	Let $\tau_{\alpha}^{\eps} = \inf \{t \in [0,T]:  \|\bar X^{\eps}(t)\| > \alpha\}$ where the infimum  is taken to be $T$ if the set is empty. Let
	$$\clr_{6,\alpha}^{\eps} = 2 \eps \sup_{0\le t \le T} \left\| \int_{(0, t\wedge \tau_{\alpha}^{\eps}]\times \X}
	h(z, \bar X^{\eps}(s-)) \ti N^{\eps^{-1}\varphi_{\eps}}(ds\, dz)\right\|.$$
	Then from \eqref{eq:eq558}, for $\eta > 0$,
	\begin{equation}
		\label{eq:953/28}
		\bar \P(\clr_{4}^{\eps}(t_0, \ups) > \eta) \le 
		\bar \P(\ti \clr_{4}^{\eps} > \eta) \le \bar \P(\clr_{6,\alpha}^{\eps} > \eta) + \bar \P((B_{\alpha}^{\eps})^c)
		\le \bar \P(\clr_{6,\alpha}^{\eps} > \eta) + \frac{\kappa_1}{\alpha}.
	\end{equation}
	For $r> 0$, write, $\clr_{6,\alpha}^{\eps} = \clr_{7,\alpha}^{\eps,r}+ \clr_{8,\alpha}^{\eps,r}$
	where
	$$\clr_{7,\alpha}^{\eps,r} = 2 \eps \sup_{0\le t \le T} \left\| \int_{(0, t\wedge \tau_{\alpha}^{\eps}]\times \X}
	h(z, \bar X^{\eps}(s-))1_{M_h(z)\le r} \ti N^{\eps^{-1}\varphi_{\eps}}(ds\, dz)\right\|,$$
	$$\clr_{8,\alpha}^{\eps,r} = 2 \eps \sup_{0\le t \le T} \left\| \int_{(0, t\wedge \tau_{\alpha}^{\eps}]\times \X}
	h(z, \bar X^{\eps}(s-))1_{M_h(z)> r} \ti N^{\eps^{-1}\varphi_{\eps}}(ds\, dz)\right\|.$$
	By standard martingale inequalities, for some $\kappa_2 \in (0, \infty)$ (independent of $\alpha,\eps, r$),
	\begin{align*}
	\bar \E(\clr_{7,\alpha}^{\eps,r})^2 &\le  \kappa_2 r (1+\alpha^2)\eps \bar \E\int_{\X_T} M_h(z) \varphi_{\eps}(s,z)\nu(dz) ds\\
	&\le \kappa_2r (1+\alpha^2)\eps\left(c(\delta,n, M_h) \int_{\X}M_h(z)\nu(dz) + \delta\right).
\end{align*}
Thus, for each $r > 0$ and $\alpha > 0$,
$\clr_{7,\alpha}^{\eps,r} \to 0$ in probability as $\eps \to 0$.
Also, for every $\delta > 0$,
\begin{align*}
\bar \E(\clr_{8,\alpha}^{\eps,r}) &\le 8(1+\alpha) \E\int_{\X_T} M_h(z)1_{M_h(z)>r} \varphi_{\eps}(s,z)\nu(dz) ds\\
&\le 8(1+\alpha)\left(c(\delta,n, M_h) \int_{\X}M_h(z)1_{M_h(z)>r}\nu(dz) + \delta\right).
\end{align*}
Since $M_h \in \cll^1(\nu)$, we have that for each $\alpha > 0$, $\sup_{\eps>0}\bar \E(\clr_{8,\alpha}^{\eps,r})$
converges to $0$ as $r \to \infty$.  Combining the above estimates, for each $\alpha > 0$
$$\clr_{6,\alpha}^{\eps} \to 0 \mbox{ in probability as } \eps \to 0.$$
Using this observation in \eqref{eq:953/28}, we have that for each $t_0 \in (0,T)$ and $\ups \in (0, t_0)$
\begin{equation}
	\label{eq:1037/29}
\clr_{4}^{\eps}(t_0, \ups) \mbox{ and } \ti \clr_{4}^{\eps} \mbox{ converge to }
0 \mbox{ in probability as } \eps \to 0.
\end{equation}
Combining this with \eqref{eq:1017/29}, \eqref{eq:1020/29}, \eqref{eq:1021/29} and \eqref{eq:949/29},  we see that 
for every $t_0 \in (0,T)$
$$
\sup_{t \in (t_0, T]}\|\cls_1^{\eps}(t)\| \to 0 \mbox{ in probability as } \eps \to 0.$$
Thus, from \eqref{eq:1023/29}
$$
\sup_{t \in [0, T]}\|\cls_1^{\eps}(t)\| \to 0 \mbox{ in probability as } \eps \to 0.$$
Combining this with \eqref{eq:1027/29} and \eqref{eq:1248/28} we see that
\begin{equation}
	\label{eq:1112/29}
	\bar Y^{\eps} - \bar X^{\eps} \mbox{ converges to } 0 \mbox{ in probability in } D([0,T]:\R^d).
\end{equation}
	Also, since $\|\bar Y^{\eps} - \ti X^{\eps}\|_{*,T} = \frac{1}{2} \tilde \clr_4^{\eps}$, from \eqref{eq:1037/29} we have that
	\begin{equation}
		\label{eq:1113/29}
		\bar Y^{\eps} - \ti X^{\eps} \mbox{ converges to } 0 \mbox{ in probability in } D([0,T]:\R^d).
	\end{equation}
	The above two displays together with the tightness of $\{\ti X^{\eps}\} _{\eps > 0}$ established earlier shows that
	$\{\bar X^{\eps}\} _{\eps > 0}$ is tight.
	Suppose that $(\bar X^{\eps}, \varphi_{\eps})$ converges weakly along a subsequence to
	$(X^0, \ti \varphi)$. Note that $\varphi$ and $\ti \varphi$ must have the same distribution.  Without loss of generality we can assume that convergence is almost sure and it holds along the full 
	sequence.  In fact we  can assume that
	$(\bar X^{\eps},\ti X^{\eps}, \varphi_{\eps})$ converges a.s. to $( X^{0},X^{0}, \ti \varphi)$.
	Then from Lemma \ref{lem:cgce1050}, for all $t \in [0,T]$, 
	$$\int_{\X_t} h(z,\bar X^{\eps}(s)) \varphi_{\eps}(s,z) \nu(dz) ds \to \int_{\X_t} h(z, X^{0}(s)) \ti \varphi(s,z) \nu(dz) ds\; \mbox{ a.s. }$$
	This, along with \eqref{eq:1115/29} shows that $X^0$ solves
	$$ X^{0}(t) = \int_{\X_t} h(z, X^{0}(s)) \ti \varphi(s,z) \nu(dz) ds, \, t \in [0,T],\; \mbox{a.s.}$$
	Since the above equation has a unique solution and $\ti \varphi$ and $\varphi$ have the same distribution we have that
	$X^0$ and $\bar X^0$ have the same distribution, where $\bar X^0$ was defined in \eqref{eq:eq647}. Thus we have shown that $\bar X^{\eps}$ converges in distribution to
	$\bar X^0$.  The result follows.
\end{proof}
\subsection{Completing the Proof of Theorem \ref{thm:main}.}
In view of Theorem \ref{Thm:LDP01} it suffices to check that $\{\clg^{\eps}, \eps > 0\}$, introduced in Condition \ref{cond:uniqsolv} and $\clg^0$ introduced at the beginning of Section \ref{sec:five} satisfy Condition \ref{cond:gensuff}.  Part (a) of the condition is immediate from Proposition \ref{prop:parta}
while part (b) is a consequence of Proposition \ref{prop:verifyb}. \qed

\section*{Acknowledgements}

AB  has been supported in part by the National Science Foundation (DMS-1004418, DMS-1016441, DMS-1305120) and the Army Research
 Office (W911NF-10-1-0158, W911NF-14-1-0331). Part of this work was carried out when PN was visiting University of North Carolina at Chapel Hill.  Support and hospitality of the STOR Department at UNC is gratefully acknowledged.

{\sc
\bigskip
\noi
A. Budhiraja\\
Department of Statistics and Operations Research\\
University of North Carolina\\
Chapel Hill, NC 27599, USA\\
email: budhiraj@email.unc.edu
\skp

\noi
P. Nyquist\\
Division of Applied Mathematics\\
Brown University\\
Providence, RI 02912, USA\\
email: pierre\_nyquist@brown.edu

\skp

}



\end{document}